\documentclass{article}

\usepackage[english]{babel}
\usepackage{amssymb}
\usepackage{amsmath}
\usepackage{amsthm}
\newtheorem{theorem}{Theorem}
\newtheorem{lemma}{Lemma}
\newtheorem{corollary}{Corollary}
\newtheorem{proposition}{Proposition}
\newtheorem{example}{Example}
\theoremstyle{remark}
\newtheorem{remark}{Remark}
\theoremstyle{definition}
\newtheorem{definition}{Definition}
\usepackage{color}
\usepackage{graphicx}
\usepackage{placeins}
\usepackage{caption}
\usepackage{float}
\usepackage{multirow}
\usepackage{booktabs}
\usepackage{subfig}
\usepackage[style=english]{csquotes}
\usepackage{enumerate}
\usepackage{manfnt}

\usepackage{amssymb}
\usepackage{enumerate}

\newcommand{\E}{\mathbb{E}} 
\newcommand{\Var}{\operatorname{Var}} 
\newcommand{\Cov}{\operatorname{Cov}} 
\newcommand{\corr}{\operatorname{corr}} 
\newcommand{\rr}{{\mathbb R}}

\begin{document}

\begin{center}
\large{\textbf{Intermittency of Superpositions of Ornstein-Uhlenbeck Type Processes}}\\
\bigskip
\normalsize{Danijel Grahovac$^{1}$, Nikolai N. Leonenko$^{2}$, Alla Sikorskii$^{3}$, Irena Te\v{s}njak$^{3}$}\footnote{E-mail addresses: dgrahova@mathos.hr (D. Grahovac);  LeonenkoN@cardiff.ac.uk (N.N. Leonenko); sikorska@stt.msu.edu (A. Sikorskii); tesnjaki@stt.msu.edu (I. {Te\v{s}njak})}\\
\bigskip
\small{\textit{$^{1}$Department of Mathematics, University of Osijek, Croatia; $^{2}$Cardiff School of Mathematics, Cardiff University, UK; $^{3}$Department of Statistics and Probability, Michigan State University, USA\\}}
\end{center}
\bigskip

\textbf{Abstract:}
The phenomenon of intermittency has been widely discussed in physics literature. This paper provides a model of intermittency based on L\'evy driven Ornstein-Uhlenbeck (OU) type processes. Discrete superpositions of these processes can be constructed to incorporate non-Gaussian marginal distributions and long or short range dependence. While the partial sums of finite superpositions of OU type processes obey the central limit theorem, we show that the partial sums of a large class of infinite long range dependent superpositions are intermittent. We discuss the property of intermittency and behavior of the cumulants for the superpositions of OU type processes.

\medskip

\textbf{Keywords:Ornstein-Uhlenbeck type  processes, intermittency, long range dependence, weak convergence}

\section{Introduction}

The phenomenon of intermittency has been widely discussed in physics literature (see for example \cite{bertini1995stochastic,fujisaka1984theory,molchanov1991ideas,woyczynski1998burgers,zel1987intermittency} and \cite[Chapter 8]{frisch1995turbulence}). The term is used to describe models exhibiting high degree of variability and enormous fluctuations which escape from the scope of the usual limit theory. Terms multifractality, separation of scales, dynamo effect are often used interchangeably with intermittency. For a formal definition of intermittency appearing in the theory of stochastic partial differential equations (SPDE) we follow \cite{carmona1994parabolic} and \cite[Chapter 7]{khoshnevisan2014analysis}. There, a nonnegative random field $\{\psi_t(x), t\geq 0, x \in \mathbb{R} \}$ stationary in parameter  $x$  is said to be intermittent if the function $k \mapsto \gamma(k)/k$ is strictly increasing on $[2,\infty)$ where $\gamma(k)$ is the $k$-th moment Lyapunov exponent of $\psi$ defined by
\begin{equation}\label{SPDE_ver}
    \gamma(k)=\lim_{t\to\infty} \frac{\log \E\left(\psi_t(x)\right)^k}{t},
\end{equation}
assuming the limit exists and is finite. This approach to intermittency is tailored for the analysis of SPDE and characterizes fields with progressive growth of moments.

 To compare intermittency to a slower growth of moments, consider  the sum
$
\phi_n=\sum _{i=1}^{n} \xi _i,
$
 where $\xi_i$ are positive  independent identically distributed (iid) random variables with finite moments. The $k$-th moment of $\phi _n$  grows as $n^k(\E \xi_1)^k$, therefore
\begin{equation*}
\gamma (k)=\lim_{n\to \infty }\frac {k\log n+k\log \E \xi _1 }{ n}
=0
\end{equation*}
for all $k\ge 1$. With the appropriate centering and norming, the classical central limit theorem holds.

In contrast, for a sequence of products of positive random variables $\psi_n= \prod_{i=1}^n \xi_i$
\begin{equation*}
    \gamma(k)=\lim_{n\to\infty} \frac{\log \E \psi_n^k}{n} = \log \E \xi_1^k.
\end{equation*}
If $\xi_i$ are not constant a.s., then from Jensen's inequality it follows that for $l>k$
\begin{equation*}
   \E \xi_1^k < \left(\E \xi^l \right)^{\frac{k}{l}},
\end{equation*}
showing that $\gamma(k)/k$ is strictly increasing.  The wild growth of moments of $\psi_n$ provides the main heuristic argument that intermittency implies unusual limiting behavior. A formal argument showing that under some assumptions intermittency implies large peaks in the space coordinate of the random field can be found in \cite{khoshnevisan2014analysis}, some ideas of which will be used later in this paper.

By far the most investigated model exhibiting intermittent behavior is the parabolic Anderson model (see \cite{gartner2010intermittency,gartner2006intermittency,gartner2007geometric,gartner1990parabolic}). In this paper we consider models provided by the partial sums of discrete superpositions of L\'evy driven Ornstein-Uhlenbeck (OU) type processes. While models based on L\'evy flights describe the position of particle, models given by OU dynamics describe the velocity of particle trapped in a field generated by quadratic potential (\cite{eliazar2005levy}). Applications of L\'evy-driven OU type processes include financial econometrics \cite{BNS_01,LPS_VG,Li}, fluid dynamics \cite{solomon1993observation}, plasma physics \cite{chechkin2002fractional} and biology \cite{Ric}. The stochastic model discussed in this paper provides another example of intermittency model based on the velocity (see \cite[Section 8.5]{frisch1995turbulence}). First, we modify the preceding definition of intermittency to tailor it to the analysis of sequences of partial sum processes. In the case of finite superpositions we show that the central limit theorem holds. In the case of infinite long range dependent superpositions, we show that the growth of cumulants is such that the partial sum process is intermittent. The appendix contains examples that fit our assumptions which cover, to our knowledge, all the examples with tractable distributions of superpositions.

\section{Intermittency}

For a process $\{Y(t),\, t \geq 0\}$, denote
\begin{equation*}
    \overline{q} = \sup \{ q >0 :\E|Y(t)|^q < \infty  \ \forall t\}.
\end{equation*}
Our definition of intermittency is based on the version of Lyapunov exponent that replaces $t$ in the denominator of \eqref{SPDE_ver} with $\log t$.  For a stochastic process $\{Y(t),\, t\ge 0\}$,  define the {\emph {scaling function}} at point $q \in [0,\overline{q})$ as
\begin{equation}\label{deftau}
    \tau(q) = \lim_{t\to \infty} \frac{\log \E |Y(t)|^q}{\log t},
\end{equation}
assuming the limit exists and is finite for every $q \in [0,\overline{q})$. Objects similar to the scaling function \eqref{deftau} appear in the theory of multifractal processes (see e.g. \cite{GL2015}), however, there are some important differences \cite{KX}. The following proposition gives some properties of $\tau$.

\begin{proposition}\label{propertiesoftau} The scaling function $\tau $ defined by \eqref{deftau} has the following properties:
\begin{enumerate}[(i)]
  \item $\tau$ is non-decreasing and so is $q \mapsto \tau(q)/q$;
  \item $\tau$ is convex;
  \item if for some $0<p<r<\overline{q}$, $\tau(p)/p < \tau(r)/r$, then there is a $q \in (p,r)$ such that $\tau(p)/p < \tau(q)/q < \tau(r)/r$.
\end{enumerate}
\end{proposition}

\begin{proof}
(i) For $0\leq q_1 < q_2 < \overline{q}$ Jensen's inequality implies
\begin{equation*}
   \E |Y(t)|^{q_1} =\E \left( |Y(t)|^{q_2} \right)^{\frac{q_1}{q_2}} \leq  \left(\E |Y(t)|^{q_2} \right)^{\frac{q_1}{q_2}}
\end{equation*}
and thus
\begin{equation*}
    \tau(q_1) \leq \frac{q_1}{q_2} \tau(q_2)
\end{equation*}
proving part (i).

(ii) Take $0\leq q_1 < q_2 < \overline{q}$ and $w_1, w_2 \geq 0$ such that $w_1+w_2=1$. It follows from H\"older's inequality that
\begin{equation*}
   \E |Y(t)|^{w_1 q_1 + w_2 q_2} \leq \left(\E |Y(t)|^{q_1} \right)^{w_1} \left( \E |Y(t)|^{q_2} \right)^{w_2}.
\end{equation*}
Taking logarithms, dividing by $\log t$  for $t>1$ and letting $t\to \infty$ we have
\begin{equation*}
    \tau(w_1 q_1 + w_2 q_2) \leq w_1 \tau(q_1) + w_2 \tau(q_2).
\end{equation*}

(iii) This is clear since $q \mapsto \tau(q)/q$ is continuous by (ii).
\end{proof}

We now define intermittency for a stochastic process and for a sequence of random variables by using the corresponding partial sum process.

\begin{definition}
A stochastic process $\{Y(t),\, t \geq 0\}$ is \emph{intermittent} if there exist $p, r \in (0,\overline{q})$ such that
\begin{equation*}
    \frac{\tau(p)}{p} < \frac{\tau(r)}{r}.
\end{equation*}
\end{definition}

Later in the paper, we will investigate intermittency of a stationary sequence of random variables $\{Y_i$, $i\in \mathbb{N\}}$ with finite mean. In this sense, intermittency will be considered as intermittency of the centered partial sum process
\begin{equation*}
    S(t) = \sum_{i=1}^{\lfloor t \rfloor} Y_i - \sum_{i=1}^{\lfloor t \rfloor} \E Y_i , \ t \geq 0.
\end{equation*}

Proposition \ref{propertiesoftau}(i) shows that the function $q \mapsto \tau(q)/q$ is always non-decreasing. What makes the process intermittent is the existence of points of strict increase. In section \ref{Section4}, we connect this property to the limiting behavior of cumulants of partial sums of superpositions of Ornstein-Uhlenbeck type processes. We show that while the partial sums of finite superpositions obey the central limit theorem, partial sums of  infinite long-range dependent superpositions provide examples of intermittent processes.

\section{Ornstein-Uhlenbeck type processes}

Ornstein-Uhlenbeck (OU) type process is the solution of the  stochastic
differential equation
\begin{equation}\label{OU}
dX(t)=-\lambda X(t) dt +dZ(\lambda t),\quad t\ge 0,
\end{equation}
where $\lambda >0$, and  $Z(t), t\ge 0$ is a L\'evy process.  The
process $Z $ is termed the background driving L\'{e}vy
process (BDLP) corresponding to the process $Y$.
The strong stationary solution of this equation exists if and only if
\begin{equation*}
\E \log \left( 1+\left\vert Z\left( 1\right) \right\vert
\right) <\infty .
\end{equation*}%
See  \cite{Sato_99} for a detailed discussion of OU type processes driven by L\'evy noise and their properties. The solution of \eqref{OU} is given by
\begin{equation}\label{OUsolution}
X(t)=e^{-\lambda t}X(0)+\int _0^t e^{-\lambda (t-s)}dZ(\lambda s),
\end{equation}
where the initial condition $X(0)$ is independent of the process $Z$.
Equation \eqref{OUsolution} specifies the unique (up to indistinguishability) strong solution of equation \eqref{OU} \cite{Sato_99}. The meaning of the stochastic integral in \eqref{OUsolution}
was detailed in \cite[p.214]{Applebaum_2004}.

The scaling in equation \eqref{OU} is such that the marginal distribution of the solution does not depend on $\lambda $, and the law of L\'evy process is determined uniquely by the distribution of $Y$ through the relation of the cumulant transforms. Let
\begin{equation*}
\kappa (z)=C\left\{ z;X\right\} =\log \E \exp \left\{ izX\right\}
,\quad z\in \mathbb{R}
\end{equation*}%
be the cumulant transform of a random variable $X$, and
\begin{equation*}
\kappa _m=(-i)^m\frac{d^m}{dz^m}\kappa (z)\vert _{z=0},\,\,m\ge 1
\end{equation*}
be the cumulant of order $m$ of $X$.

 The cumulant transforms of $X(t)$ and $Z(1) $ are related by
\begin{equation*}
C\left\{ z;X\right\} =\int_{0}^{\infty }C\left\{ e^{-s}z;Z(1)
\right\} ds=\int_{0}^{{}z}C\left\{ \xi ;Z(1) \right\} \frac{d\xi
}{\xi }
\end{equation*}
and
\begin{equation*}
C\left\{ z;Z(1) \right\} =z\frac{\partial C\left\{ z;X\right\} }{
\partial z}.
\end{equation*}
By specifying the appropriate BLDP, OU type processes with given self-decomposable marginal distributions can be obtained. These distributions include normal, Gamma, inverse Gaussian, Student's t, and many others. If the second moment is finite, the correlation function is exponential:
\begin{equation*}
\corr (X(t), X(s))=e^{-\lambda (t-s)},\,\, t\ge s\ge 0.
\end{equation*}

\section{Discrete superpositions of Ornstein-Uhlenbeck type processes}

Superpositions of OU type processes, or supOU processes for short, were introduced in \cite{BN98,BN_01,BN_05}, see also  \cite{BNL_05a,BNS_01,fasen2007extremes}, among others.
We define the superpositions under the following condition:

{\bf {(A)}} Let $X^{(k)}(t),\,k\geq 1$ be the sequence of independent stationary processes such
that each $X^{(k)}(t)$ is the stationary solution of the equation
\begin{equation}
dX^{(k)}(t)=-\lambda _kX^{(k)}(t) dt +dZ^{(k)}(\lambda _kt),\quad t\geq
0,  \label{sup1}
\end{equation}%
in which the L\'{e}vy processes $Z^{(k)}$ are independent,  and $\lambda _k>0 $  for all $k\ge 1$.
Assume that the self decomposable distribution of $X^{(k)}$ has finite moments of order $p \ge 2$ and that cumulants of orders $2, \dots , p$ of $X^{(k)}$  are proportional to some parameter $\delta _k $ of the distribution of $X^{(k)}$.

Define the superposition of OU processes, either finite for
an integer $K\ge 1$
\begin{equation}\label{supfinite}
X_{K}(t)=\sum_{k=1}^{K}X^{(k)}(t),\,\,t\in \mathbb{R}
\end{equation}%
or infinite
\begin{equation}\label{supinfinite}
X_{\infty }(t)=\sum_{k=1}^{\infty }X^{(k)}(t),\,\,t\in \mathbb{R}.
\end{equation}
The construction with infinite superposition is well-defined in the sense of
mean-square or almost-sure convergence provided that the following condition holds:
\begin{equation*}
{\text {\bf (B)}}\quad \sum_{k=1}^{\infty }\E X^{(k)}(t)<\infty \text{ and } \sum_{k=1}^{\infty.
}VarX^{(k)}(t)<\infty .
\end{equation*}
Although assumption (A) may seem restrictive, it is actually easy to show that it is satisfied for many examples with tractable distributions of superpositions. The appendix  provides a number of examples where both assumptions (A) and (B) are satisfied. These examples include Gamma, inverse Gaussian and other well known distributions. Their superpositions have the marginal distributions that belong to the same class as the marginal distributions of the components of superposition.

In the case of finite superposition, the covariance function of the resulting process is
\begin{equation*}
R_{X_{K}}(t)=\Cov(X_{K}(0),X_{K}(t))=\sum_{k=1}^{K}\Var(X^{(k)}(t))e^{-\lambda _kt},
\end{equation*}%
and the finite superposition is a short-range dependent process since the correlation function is integrable.

In the case of infinite superposition, the covariance function is
\begin{equation*}
R_{X_{\infty }}(t)=\Cov(X_{\infty }(0),X_{\infty }(t))=\sum_{k=1}^{\infty
}\Var(X^{(k)}(t))e^{-\lambda _kt},
\end{equation*}%
and under the condition (A) the variance of $X^{(k)}(t)$
is proportional to $\delta_k,$ that is
\begin{equation*}
\Var(X^{(k)}(t))=\delta _{k}C_{2},
\end{equation*}%
where constant $C_{2}$ does not depend on $k$ and reflects parameters of the
marginal distribution of $X^{(k)}$. If one chooses
\begin{equation*}
\delta _{k}=k^{-(1+2(1-H))},\,\,\frac{1}{2}<H<1, \quad
\lambda _k=\lambda /k
\end{equation*}
 for some $\lambda >0$, then
\begin{equation}\label{infcov}
R_{X_{\infty }}(t)=C_{2}\sum_{k=1}^{\infty }\frac{1}{k^{1+2(1-H)}}e^{-\lambda t/k}.
\end{equation}
Lemma below shows that the correlation function \eqref{infcov} is not integrable for the chosen parameters $\delta _k$ and $\lambda _k$, thus the process obtained via infinite superposition exhibits long-range dependence.
\begin{lemma}\label{Lemma1}
For the infinite superposition \eqref{supinfinite} of OU type processes that satisfy condition (A) with $p=2$  and condition (B), the covariance function of $X_{\infty }(t)$  given by \eqref{infcov} with  $\lambda ^{(k)}=\lambda /k$ and
$\delta _{k}=k^{-(1+2(1-H))}$, $\frac{1}{2}<H<1$, can
be written as
\begin{equation*}
R_{X_\infty }(t)=\frac{L(t)}{t^{2(1-H)}}, \quad t>0
\end{equation*}%
where $L$ is a slowly varying at infinity function.
\end{lemma}

\begin{proof}
The proof of this lemma is essentially the same as the proofs presented for particular cases of superpositions of OU processes in \cite{LTauf,LPS_d}. We provide it here for completeness and for the remark that follows. The remark will be used for proofs later in the paper.
Let
\begin{equation*}
L(t)=C_{2}\left (\lambda t\right)^{2(1-H)}\sum_{k=1}^{\infty }\frac{1}{k^{1+2(1-H)}}e^{-\lambda t/k}.
\end{equation*}%
Estimate the sum appearing in the expression for $L$ as follows:
\begin{equation*}
\int_{1}^{\infty }\frac{e^{-\lambda t/u}}{u^{1+2(1-H)}}du\leq \sum_{k=1}^{\infty }%
\frac{1}{k^{1+2(1-H)}}e^{-\lambda t/k}\leq \int_{1}^{\infty }\frac{e^{-\lambda t/u}}{%
u^{1+2(1-H)}}du+e^{-\lambda t}.
\end{equation*}%
Transform the variables $\lambda t/u=s$ to get
\begin{equation*}
C_2\int_{0}^{\lambda t}e^{-s}s^{2(1-H)-1}ds\leq L(t)\leq
C_2\int_{0}^{\lambda t}e^{-s}s^{2(1-H)-1}ds+C_2e^{-\lambda t}\left (\lambda t\right)^{2(1-H)}.
\end{equation*}%
Since
\begin{equation*}
\int_{0}^{\lambda t}e^{-s}s^{2(1-H)-1}ds\rightarrow \Gamma (2(1-H))
\end{equation*}%
as $t\rightarrow \infty $, it follows that $\lim_{t\rightarrow \infty
}L(tv)/L(t)=1$ for any fixed $v>0$.

\end{proof}

\begin{remark}
\label{bounded}
From proof of Lemma \ref{Lemma1}
\begin{equation*}
\begin{split}
L([Nt])&\le C_2 \int _0^{\lambda [Nt]} e^{-s} s^{2(1-H)-1}ds+C_2 e^{-\lambda [Nt]} (\lambda [Nt])^{2(1-H)} \\
&\le C_2 \Gamma (2(1-H))+C_2 e^{-2(1-H)}(2(1-H)])^{2(1-H)}
\end{split}
\end{equation*}
for all $N\ge 1$ and $t\in [0, 1]$ since the function $x^{2(1-H)}e^{-x}$ is bounded (attains its maximum at $x=2(1-H)$).
Also from the proof of Lemma \ref{Lemma1}
\begin{equation*}
L(N)\ge C_2 \int _0^{\lambda N} e^{-s} s^{2(1-H)-1}ds\ge  C_2 \int _0^{\lambda } e^{-s} s^{2(1-H)-1}ds
\end{equation*}
for all $N\ge 1$. Also note that $L(0)=0$.  Therefore the ratio $L([Nt])/L(N)$ is bounded uniformly in $ N\ge 1$ and $t\ge 0$.
\end{remark}

\section{Limit distributions of partial sums of superpositions of supOU processes}\label{Section4}

For $t>0$, consider partial sum processes
\begin{equation}\label{finitepartial}
S_K(t)=\sum _{i=1}^{[t]}X _K(i)
\end{equation}
and
\begin{equation}\label{infpartial}
S_{\infty } (t)=\sum _{i=1}^{[t]}X _{\infty }(i).
\end{equation}

We begin with the limit distribution of the partial sum process for the finite superposition. The asymptotic normality in this case is easy to prove using the strong mixing property of OU processes established in  \cite {JVDM_05,M_04}. Previously asymptotic normality of partial sums was reported for inverse Gaussian and gamma finite superpositions \cite{LPS_VG}. The result below is a straightforward generalization to a more general class of processes.

\begin{theorem}
For a fixed integer $K\ge 1$,  let $X_K$ be defined by \eqref{supfinite}, where stationary OU type processes $\{X^{(k)},\,k=1, \dots , K\}$ defined by \eqref{sup1}  are independent and $\E |X^{(k)}|^{2+d}<\infty $ for  some $d >0$ and all $k=1,\dots , K$. Then the partial sums process \eqref{finitepartial}, centered and appropriately normed, converges to the Brownian motion
$$\frac {1}{c_KN^{1/2}}\biggl(S_K([Nt])-\E S_m([Nt])\biggr)\to B(t),\quad t\in[0,1],$$
as $N\to \infty $ in the sense of weak convergence in Skorokhod space $D[0,1]$. The norming constant $c_K$  is given by
\begin{equation*}
c_K=\biggl(\sum_{k=1}^K \Var \left (X^{(k)} \right)\frac {1-e^{-\lambda ^{(k)}}} {1+e^{-\lambda ^{(k)}}}\biggr)^{1/2}.
\end{equation*}
\end{theorem}
\begin{proof}
Since  each OU process in the superposition has a finite second moment,  $\beta $-mixing (absolute regularity) for each OU process holds with the exponential rate. Namely, there exists $a_k>0$ such that the mixing coefficient $\beta_{X^{(k)}} (t)=O(e^{-a_kt})$ \cite[Theorem 4.3]{M_04}.   Denote by $\alpha ^{(k)}(t)$ the strong mixing coefficient of the process $X^{(k)}$, then from  \cite{Brad},  $2\alpha ^{(k)}(t) \le \beta ^{(k)}(t)\le D_k e^{-a_kt}$  for a constant $D_k$, for each $k=1,\dots ,m $.
A finite sum of $\alpha $-mixing processes with exponentially decaying mixing coefficients is also $\alpha $-mixing with exponentially decaying mixing coefficient, therefore weak convergence  of partial sums of the process $X_K$ in  $D[0,1]$ follows from \cite[Theorem 4.2] {D_68}.
\end{proof}
 We now proceed with the limit distribution of the partial sum process for the infinite superposition \eqref{supinfinite}. The variance of this process has been computed in \cite[Equation (5.3)]{LTauf}, however the result on the asymptotic normality of the partial sum process \cite [Theorem 3]{LTauf} was not correct. Also incorrect was statement (30) of \cite [Theorem 5]{BNL_05}. Here we provide the derivation of the variance and correct the result on the limit distribution.

\begin{lemma}\label{Lemma42}
For the infinite superposition \eqref{supinfinite} of OU type processes that satisfy condition (A) with $p=2$ and condition (B), set  $\lambda ^{(k)}=\lambda /k$ and
$\delta _{k}=k^{-(1+2(1-H))}$, $\frac{1}{2}<H<1$. Then
\begin{equation}\label{infvariance}
\Var\left( S_{\infty }([Nt])\right)=\frac {L(N) [Nt]^{2H}}{H(2H-1)}\left (1+o(1)\right) \text { as } N\to \infty,
\end{equation}
where $L$ is a slowly varying at infinity function.
\end{lemma}

\begin{proof}
Using the expression for the covariance function of the infinite superposition from Lemma \ref{Lemma1}, write
\begin{equation*}
\begin{split}
\Var\left( S_{\infty }([Nt])\right) &=\sum _{m,\,n=1}^{[Nt]}\Cov \left(X_\infty (m), X_\infty (n)\right)  \\
&=[Nt]\Var \left(X_\infty (n)\right) +2\sum _{m,\,n=1,\,m>n}^{[Nt]} \frac{L(m-n)}{(m-n)^{2(1-H)}}\\
&=C_2[Nt]\zeta (1+2(1-H))+2\sum _{j=1}^{[Nt]-1}\left([Nt]-j\right)\frac {L(j)}{j^{2(1-H)}},
\end{split}
\end{equation*}
where $\zeta (\cdot )$ is Riemann's zeta function. The sum appearing in the expression for the variance
\begin{equation*}
\sum _{j=1}^{[Nt]-1}\left([Nt]-j\right)\frac {L(j)}{j^{2(1-H)}}
\end{equation*}
 is a Riemann sum for the following integral:
\begin{equation*}
\int _0^1 ([Nt]-[Nt]u)\frac{L([Nt]u)}{([Nt]u)^{2(1-H)}}[Nt]du=[Nt]^{2H}\int _0^1 (1-u)u^{2H-2}L([Nt]u)du.
\end{equation*}
Consider the integral
\begin{equation*}
\int _0^1 u^{2H-2}L([Nt]u)du=\frac {1}{[Nt]^{2H-1}}\int _0^{[Nt]}v^{2H-2}L(v)dv,
\end{equation*}
and apply Karamata's theorem \cite[Theorem 2.1]{Resnick} to get
\begin{equation*}
\int _0^{[Nt]}v^{2H-2}L(v)dv=\frac {L(N)[Nt]^{2H-1}}{2H-1}(1+o(1))
\end{equation*}
as $N\to \infty $. Similarly,
\begin{equation*}
\int _0^1u^{2H-1}L([Nt]u)du=\frac {L(N)}{2H}(1+o(1))
\end{equation*}
as $N\to \infty $, and therefore
\begin{equation*}
\int _0^1 ([Nt]-[Nt]u)\frac{L([Nt]u)}{([Nt]u)^{2(1-H)}}[Nt]du=\frac {L(N)[Nt]^{2H}}{2H(2H-1)}(1+o(1)).
\end{equation*}
For $\frac{1}{2}<H<1$, the second term in the expression for the variance of $S_\infty ([Nt])$ dominates the first, and  \eqref{infvariance} follows.

\end{proof}

In order to characterize the limit distribution of the partial sums of the infinite superpositions, we use the representation of the discretized stationary OU process as a first order autoregressive sequence
\begin{equation}\label{ar}
X^{(k)}(i)=e^{-\lambda _k}X^{(k)}(i-1)+W^{(k)}(i),
\end{equation}
where $W^{(k)}(i)$ is independent of $X^{(k)}(j)$ for all $j<i$. Denote by $\rho _k=e^{-\lambda _k}$. The following lemma provides a useful representation of the partial sum process for the infinite superposition.

\begin{lemma}\label{Lemma:representation}
The centered partial sum of the superposition of processes that satisfy  condition (A) with $p=2$ and condition (B) with $\lambda ^{(k)}=\lambda /k$ and
$\delta _{k}=k^{-(1+2(1-H))}$, $\frac{1}{2}<H<1$,  can be written as
\begin{equation}\label {representation}
S_\infty ([Nt])-\E S_\infty ([Nt]=\sum _{k=1}^\infty b^{(k)}_{[Nt]}\tau ^{(k)} (0)+ \sum _{j=1}^{[Nt]}\sum _{k=1}^\infty a^{(k)}_{[Nt]-j}V^{(k)}(j),
\end{equation}
where $\tau ^{(k)}(0)$,  $V^{(k)}(j)$ are independent for different $k$, for each $k$ $V^{(k)}(j)$ are independent for different $j$ and also independent of $\tau ^{(k)}(0)$. The series in \eqref{representation} converge almost surely, and the  coefficients are given by
\begin{equation}\label{bcoefficient}
b^{(k)}_{[Nt]}=\sum _{i=1}^{[Nt]}\rho_k ^i=\frac{\rho _k (1-\rho _k^{[Nt]})}{1-\rho _k},
\end{equation}
and
\begin{equation}\label{acoefficient}
a^{(k)}_{[Nt]-j}=\sum _{i=0}^{[Nt]-j}\rho _k^i=\frac {1-\rho _k^{[Nt]-j+1}}{1-\rho _k}.
\end{equation}
\end{lemma}

\begin{proof}
 Center the variables
\begin{equation*}
\tau^{(k)}(i)=
X^{(k)}(i)-\E X^{(k)}(i), \quad
V^{(k)}(i)= W^{(k)}(i)- \E W^{(k)}(i)
\end{equation*}
to arrive at centered version of \eqref{ar}
\begin{equation}\label{arcentered}
\tau^{(k)}(i)=\rho_k \tau^{(k)}(i-1)+V^{(k)}(i).
\end{equation}
Iterate \eqref{arcentered} to obtain
\begin{equation*}
\tau^{(k)}(i)=\rho _k^i \tau ^{(k)} (0)+\sum _{j=1}^{i}\rho_k ^{i-j}V^{(k)}(j).
\end{equation*}
Now the partial sum of $\tau ^{(k)}$ can be written
\begin{equation*}
\begin{split}
\sum _{i=1}^{[Nt]}\tau^{(k)}(i)&=\tau ^{(k)} (0)\sum _{i=1}^{[Nt]}\rho_k ^i+\sum _{i=1}^{[Nt]}\sum _{j=1}^{i}\rho_k ^{i-j}V^{(k)}(j)\\
&=\tau ^{(k)} (0)\sum _{i=1}^{[Nt]}\rho_k ^i+\sum _{j=1}^{[Nt]}V^{(k)}(j)\sum _{i=j}^{[Nt]}\rho_k ^{i-j}\\
&=\tau ^{(k)} (0)\sum _{i=1}^{[Nt]}\rho_k ^i+\sum _{j=1}^{[Nt]}V^{(k)}(j)\sum _{m=0}^{[Nt]-j}\rho_k ^{m}\\
&=b^{(k)}_{[Nt]}\tau ^{(k)} (0)+ \sum _{j=1}^{[Nt]}a^{(k)}_{[Nt]-j}V^{(k)}(j),
\end{split}
\end{equation*}
where the coefficients are given by \eqref{bcoefficient} and \eqref{acoefficient}.
Note that for different $j$, $V^{(k)}(j)$ are independent due to \eqref{ar}, and they are also independent of $\tau ^{(k)}(0)$. For different $k$, independence follows from the independence of OU type processes $X^{(k)}$. Summing with respect to $k$ completes the derivation of \eqref{representation}, provided that the series in \eqref{representation} converge almost surely. Series convergence holds because the terms have zero mean, and the series of second moments converge. The latter is shown as follows. Series of the second moments for the first term series in \eqref{representation} is
\begin{equation*}
\begin{split}
&\sum _{k=1}^\infty (b^{(k)}_{[Nt]})^2 \E (\tau ^{(k)} (0))^2=C_2 \sum _{k=1}^\infty (b^{(k)}_{[Nt]})^2 \delta _k\\
&= C_2 \sum _{k=1}^\infty \sum _{j,i=1}^{[Nt]}\rho _k ^{i+j}\delta _k= \frac{1}{\lambda ^{2(1-H)}}\sum _{j,i=1}^{[Nt]}\frac {L(i+j)}{(i+j)^{2(1-H)}}.
\end{split}
\end{equation*}
The sum can be viewed as a Riemann sum for the double integral:
\begin{equation*}
\begin{split}
& \frac{1}{[Nt]^2}\sum _{j,i=1}^{[Nt]}\frac {L(i+j)}{(i+j)^{2(1-H)}}=\int _0^1\int_0^1
\frac{L([Nt](x+y))}{([Nt](x+y))^{2(1-H)}}dxdy (1+o(1)) \\
&=\frac {L(N)}{[Nt]^{2(1-H)}}\int _0^1\int_0^1 \frac {dx\, dy}{(x+y)^{2(1-H)}} (1+(o(1))
\end{split}
\end{equation*}
as $N\to \infty $. The last equality is justified using Karamata's theorem as in Lemma \ref{Lemma42}, or by considering
\begin{equation*}
L(N)\int _{x=\epsilon}^1 \int _{y=0}^{1}\frac{L([Nt](x+y))}{L(N)}\frac{dx\,dy}{(x+y)^{2(1-H)}}
\end{equation*}
 and using Remark \ref{bounded} and the dominated convergence theorem.
Therefore the variance of the first series in \eqref{representation} is of the order $L(N)N^{2H}$.

For the second term in \eqref{representation}, the series of second moments is
\begin{equation*}
\begin{split}
 \sum _{j=1}^{[Nt]} \sum _{k=1}^\infty \left(a^{(k)}_{[Nt]-j}\right)^2 \E \left( V^{(k)}(j)\right)^2= \sum _{j=1}^{[Nt]} \sum _{k=1}^\infty \left (\sum _{i=0}^{[Nt]-j}\rho _k^i\right)^2 (1-\rho _k^2)C_2\delta _k,
\end{split}
\end{equation*}
since $\E (V^{(k)}(j))^2=(1-\rho _k^2)\E (\tau ^{(k)})^2$. The series of second moments becomes
\begin{equation*}
\begin{split}
&\sum _{j=1}^{[Nt]} \sum _{k=1}^\infty \sum _{i_1,i_2=0}^{[Nt]-j}\rho _k^{i_1+i_2}(1-\rho _k^2)C_2\delta _k\\
&= \frac {1}{\lambda ^{2(1-H)}}\sum _{j=1}^{[Nt]}\sum _{i_1,i_2=0}^{[Nt]-j} \left( \frac {L(i_1+i_2)}{(i_1+i_2)^{2(1-H)}}- \frac {L(i_1+i_2+2)}{(i_1+i_2+2)^{2(1-H)}}\right) = \frac {[Nt]^3}{\lambda ^{2(1-H)}} \times \\
& \int _{x=0}^1 \int _{y=0}^{1-x }\int _{z=0}^{1-x}
\left(\frac {L([Nt](y+z))}{([Nt](y+z))^{2(1-H)}}-\frac {L([Nt](y+z)+2)}{([Nt](y+z)+2)^{2(1-H)}}\right) dxdydz\\
& \times (1+o(1)).
\end{split}
\end{equation*}
Arguing in the same way as for the first term in  \eqref{representation}, we have
\begin{equation*}
\begin{split}
&\sum _{j=1}^{[Nt]} \sum _{k=1}^\infty \sum _{i_1,i_2=0}^{[Nt]-j}\rho _k^{i_1+i_2}(1-\rho _k^2)C_2\delta _k
= \frac {[Nt]^{2H+1} L(N)}{\lambda ^{2(1-H)}} \\
& \times  \int _{x=0}^1 \int _{y=0}^{1-x }\int _{z=0}^{1-x} \left (\frac {1}{(y+z)^{2(1-H)}}-\frac {1}{((y+z)+2/[Nt])^{2(1-H)}}\right) dxdydz\\
&= \frac {[Nt]^{2H} L(N)}{(2H-1)\lambda ^{2(1-H)}} \times  \int _{x=0}^1 \int _{y=0}^{1-x } [Nt]\Big( (y+2/[Nt])^{2H-1} - y^{2H-1} \\
& \qquad \qquad  - \left( (y+1-x+2/[Nt])^{2H-1} - (y+1-x)^{2H-1}\right) \Big) dxdy (1+o(1)).
\end{split}
\end{equation*}
It is not hard to see that the function under the  integral is increasing in $[Nt]$ and as $[Nt]\to \infty$ the limit is
\begin{equation*}
    2(2H-1) \left( y^{2H-2} - (y+1-x)^{2H-2}\right).
\end{equation*}
The monotone convergence theorem yields
\begin{equation*}
\begin{split}
&\sum _{j=1}^{[Nt]} \sum _{k=1}^\infty \sum _{i_1,i_2=0}^{[Nt]-j}\rho _k^{i_1+i_2}(1-\rho _k^2)C_2\delta _k\\
&= \frac {2 [Nt]^{2H} L(N)}{\lambda ^{2(1-H)}}  \int _{x=0}^1 \int _{y=0}^{1-x } \left( y^{2H-2} - (y+1-x)^{2H-2}\right) dxdy (1+o(1)),
\end{split}
\end{equation*}
which shows that the series in the second term converges almost surely, and that the variance of the second term has the same order as the variance of the first term, namely $L(N)[Nt]^{2H}$.
\end{proof}

The next theorem gives the asymptotic behavior of the cumulants of the partial sum process.

\begin{theorem}\label{theorem:cumulants}
The  m-th cumulant of the centered partial sum of the superposition of processes  that satisfy  condition (A) for all $p \ge 2$,  condition (B), and has $\lambda ^{(k)}=\lambda /k$ and
$\delta _{k}=k^{-(1+2(1-H))}$, $\frac{1}{2}<H<1$, has the following asymptotic behavior
\begin{equation*}
\kappa_{m, N} =D_mL(N)[Nt]^{m-2(1-H)}(1+o(1))
\end{equation*}
as $N\to \infty $, where the $D_m= C_m K$ for some positive constant $K$.
\end{theorem}

\begin{proof}
Using \eqref{representation}, the logarithm of the characteristic function of the partial sum process can be written as
\begin{equation*}
\begin{split}
&\log \E \exp \left\{ iu (S_\infty ([Nt])-\E S_\infty ([Nt])\right\}\\
&=\sum _{k=1}^\infty \log \E \exp \left\{i b^{(k)}_{[Nt]} u \tau ^{(k)} (0 )\right\}+ \sum _{j=1}^{[Nt]}\sum _{k=1}^\infty
\log \E \exp \left\{ ia^{(k)}_{[Nt]-j} u V^{(k)}(j)\right\}.
\end{split}
\end{equation*}
Under assumption (A), the logarithm of the characteristic function of $\tau ^{(k)} (0)$ can be expanded
\begin{equation*}
\log \E \exp \left\{ iu \tau ^{(k)} (0 )\right\}=\sum_{m=2}^\infty \frac{ {(iu)}^m}{m!} C_m\delta _k,
\end{equation*}
where the summation is from $m=2$ due to centering.
From \eqref{arcentered}, the logarithm of the characteristic function of $ V^{(k)}(j)$ can also be expanded as follows:
\begin{equation*}
\begin{split}
\log \E \exp \left \{iu V^{(k)}(j)\right\}&=\ E \exp \left \{iu \tau^{(k)}(i)\right\} - \E \exp \left \{iu\rho _k \tau ^{(k)}(i-1)\right\} \\
&=\sum_{m=2}^\infty \frac{ {(iu)}^m}{m!} C_m\delta _k-\sum_{m=2}^\infty \frac{ {(iu\rho _k)}^m}{m!} C_m\delta _k\\
&=\sum_{m=2}^\infty \frac{ {(iu)}^m}{m!} C_m(1-\rho _k ^m) \delta _k.
\end{split}
\end{equation*}
Therefore the $m$-th cumulant of the centered partial sum process is
\begin{equation*}
\kappa _{m, N}=C_m \sum_{k=1}^\infty  \left(b^{(k)}_{[Nt]}\right)^m \delta _k +C_m\sum _{j=1}^{[Nt]} \sum_{k=1}^\infty  \left (a^{(k)}_{[Nt]-j}\right)^m (1-\rho _k ^m) \delta _k=I+II.
\end{equation*}
Consider the first term:
\begin{equation*}
\begin{split}
I&=C_m\sum_{k=1}^\infty  \left(b^{(k)}_{[Nt]}\right)^m \delta _k=C_m\sum_{k=1}^\infty \delta _k
\left( \sum _{i=1}^{[Nt]} \rho _k ^i \right)^m\\
&=C_m\sum_{i_1,\dots , i_m=1}^{[Nt]} \sum _{k=1}^\infty \delta _k \rho _k^{i_1+\cdots +i_m}=\frac {C_m}{C_2  \lambda ^{2(1-H)}}\sum_{i_1,\dots , i_m=1}^{[Nt]}\frac {L(i_1+\cdots +i_m)}{(i_1+\cdots +i_m)^{2(1-H)}}\\
&=\frac {C_m[Nt]^m}{C_2  \lambda ^{2(1-H)}}\int _0^1 \dots \int _0^1 \frac {L([Nt](x_1+\cdots +x_m))}
{([Nt](x_1+\cdots +x_m))^{2(1-H)}}dx_1\dots dx_m\left(1+o(1)\right) \\
&= \frac {C_m[Nt]^mL(N)}{C_2  \lambda ^{2(1-H)}[Nt]^{(2(1-H)}}\int _0^1 \dots \int _0^1 \frac {dx_1\dots dx_m}{(x_1+\cdots +x_m)^{2(1-H)}}\left(1+o(1)\right),
\end{split}
\end{equation*}
where we used Remark \ref{bounded} and the dominated convergence argument for the slowly varying function. This shows that the first part of the expression for the $m$-th cumulant behaves like $L(N)[Nt]^{m-2(1-H)}$ multiplied by a constant
\begin{equation*}
D_{m, I}=\frac {C_m}{C_2 \lambda ^{2(1-H)}} \int _0^1 \dots \int _0^1 \frac {dx_1\dots dx_m}{(x_1+\cdots +x_m)^{2(1-H)}}.
\end{equation*}
Now consider the second term
\begin{equation*}
\begin{split}
II&=C_m\sum _{j=1}^{[Nt]}\sum_{k=1}^\infty  \left (a^{(k)}_{[Nt]-j}\right)^m (1-\rho _k ^m) \delta _k \\
&=
C_m\sum _{j=1}^{[Nt]}\sum_{k=1}^\infty  \left( \sum _{i=0}^{[Nt]-j} \rho_k^i \right)^m (1-\rho _k ^m) \delta _k\\
&= C_m\sum _{j=1}^{[Nt]}\sum_{k=1}^\infty \sum _{i_1,\dots , i_m=0}^{[Nt]-j} \rho _k ^{i_1+\cdots +i_m} (1-\rho _k ^m) \delta _k
=\frac {C_m}{C_2 \lambda ^{2(1-H)}}\\
& \times \sum _{j=1}^{[Nt]}\sum _{i_1,\dots , i_m=0}^{[Nt]-j}  \left (  \frac {L(i_1+\cdots +i_m)}{(i_1+\cdots i_m)^{2(1-H)} }-     \frac {L(i_1+\cdots +i_m+m)}{(i_1+\cdots i_m+m)^{2(1-H)} } \right) =\frac {C_m[Nt]^{m+1}}{C_2 \lambda ^{2(1-H)}} \\
& \times \int_{x=0}^1 \int _{y_1=0}^{1-x}\dots \int_{y_m=0}^{1-x}  \left (  \frac {L([Nt](y_1+\cdots +y_m))}{([Nt](y_1+\cdots y_m))^{2(1-H)} }- \frac{L([Nt](y_1+\cdots +y_m)+m)}{([Nt](y_1+\cdots y_m)+m)^{2(1-H)} }\right)\\
& \times  dy_1\dots dy_{m}dx\left (1+o(1)\right)
\end{split}
\end{equation*}
Remark \ref{bounded} and the dominated convergence argument yield
\begin{equation*}
\begin{split}
II&=\frac {C_m L(N) [Nt]^{m-2(1-H)+1}}{C_2 \lambda ^{2(1-H)}} \\
& \times \int_{x=0}^1 \int _{y_1=0}^{1-x}\dots \int_{y_m=0}^{1-x} \left (  \frac {1}{(y_1+\cdots y_m)^{2(1-H)} }- \frac{1}{((y_1+\cdots y_m)+m/[Nt])^{2(1-H)} }\right)\\
& \times  dy_1\dots dy_{m} dx\left (1+o(1)\right) =\frac {m C_m L(N) [Nt]^{m-2(1-H)}}{C_2 \lambda ^{2(1-H)}} \\
& \times \int_{x=0}^1 \int _{y_1=0}^{1-x}\dots \int_{y_{m-1}=0}^{1-x} \left ( (y_1+\cdots y_{m-1})^{2H-2} - (y_1+\cdots y_{m-1}+1-x)^{2H-2}\right)\\
& \times  dy_1\dots dy_{m-1} dx\left (1+o(1)\right) = D_{m, II} L(N)[Nt]^{m-2(1-H)} \left (1+o(1)\right)
\end{split}
\end{equation*}
with
\begin{equation*}
\begin{split}
D_{m, II} &= \frac {m C_m}{C_2 \lambda ^{2(1-H)}} \int_{x=0}^1 \int _{y_1=0}^{1-x}\dots \int_{y_{m-1}=0}^{1-x} \\
& \qquad \left ( (y_1+\cdots y_{m-1})^{2H-2} - (y_1+\cdots y_{m-1}+1-x)^{2H-2}\right)  dy_1\dots dy_{m-1} dx.
\end{split}
\end{equation*}
Thus the asymptotic behavior of the second term is the same as of the first term, namely $L(N)[Nt]^{m-2(1-H)}$.
\end{proof}

\begin{corollary}
Under the assumptions of Theorem \ref{theorem:cumulants}, the centered partial sum process $\{S_\infty (t)-\E S_\infty (t),\, t \geq 0\}$ is intermittent.
\end{corollary}

\begin{proof}
Let $Y([Nt])=S_{\infty}([Nt])-\E S_{\infty}([Nt])$. We show intermittency at $p=2$ and $r=4$. Using the relation between moments and cumulants it follows from Theorem \ref{theorem:cumulants} that
\begin{align*}
E|Y(N)|^2 &= \kappa_{2,N} + \kappa_{1,N}^2 = D_2 L(N) N^{2H}(1+o(1)),\\
E|Y(N)|^4 &= \kappa_{4,N} + 3 \kappa_{2,N}^2 = D_4 L(N) N^{2H+2}(1+o(1)) + 3 D_2^2 L(N)^2 N^{4H}(1+o(1)).
\end{align*}
Since $H<1$ implies $2H+2 > 4H$, we have
\begin{align*}
\tau(2) &= 2H,\\
\tau(4) &= 2H+2,
\end{align*}
and thus $\tau(2)/2 < \tau(4)/4$.
\end{proof}

Note that the behavior of moments shown in the proof implies that
\begin{equation*}
\E Y(N)^4 / (\E Y(N)^2)^2
\end{equation*}
grows to infinity as $N\to\infty$, the behavior noted by Frisch (\cite[Section 8.2]{frisch1995turbulence} as a manifestation of intermittency. Other examples of unusual growth of moments are given in \cite{sandev2015distributed} in the context of fractional diffusion. Also note that if the limit of the partial sum process for the infinite superposition existed in the sense of all finite dimensional distribution, then by the Lamperti's theorem  (see, for example, \cite[Theorem 2.1.1]{embrechts2002selfsimilar}), the norming had to be a regularly varying function of $N$. However, if $N^a$  norming is used, it can not produce all converging cumulants no matter what $a\in \rr$ is chosen. This is because the m-th cumulant of $N^aY([Nt])=N^a\left(S_{\infty}([Nt])-\E S_{\infty}([Nt])\right)$ behaves as $N^{m(a+1)-2(1-H)}$. Similar cumulant behavior was found in \cite{Terdik}, where it was noted that the existence of the limit was unlikely. Also, similar behavior of cumulants was obtained in \cite[Example 4.1]{BN_01} for a case of continuous (integrated) superpositions of OU type processes.  Of course, convergence of cumulants provides a sufficient means for proving the existence of the limit, and showing that there is no weak limit under intermittency in the usual partial sum setting remains an open problem.

Also note that for even $q$, the scaling function defined in  \eqref{deftau} in this case  is $\tau (q)=q-2(1-H)$, and
$$\frac{\tau (q)}{q}=1-\frac {2(1-H)}{q}$$ is strictly increasing in $q$. The term $-2(1-H)$ in the exponent of the asymptotic behavior of the cumulants
\begin{equation*}
\kappa_{q, N} =D_qL(N)[Nt]^{q-2(1-H)}(1+o(1))
\end{equation*}
gives the reason for both the increasing behavior of $\tau(q)/q$ and for the lack of norming that would make cumulants converge. The formal link between intermittency and lack of the limit theorems needs to be further developed for the partial sums and other sequences of stochastic processes.

\section{Appendix}\label{Examples}
The examples in this section have been discussed in \cite{BN_01,LTauf}. We briefly present them to illustrate that conditions (A) and (B) are satisfied for a number of OU type processes.

\begin{example}

The stationary Gamma OU type process $\{X(t), t\geq 0\}$ with gamma marginal distribution has the cumulant generating function
\begin{equation}\label{Ex1GammaCum}
\kappa (\zeta)=\log \E \exp \left\{ i\zeta X(t) \right\} = - \alpha \log \left(1-\frac{i\zeta}{\beta} \right) = \sum_{m=1}^{\infty} \frac{\alpha (i \zeta )^m}{m \beta^m},
\end{equation}
$\alpha >0$, $\beta >0$, $\zeta<\beta $.
If $\{X^{(k)}(t), t\geq 0\}$, $k\geq 1$ are independent stationary Gamma OU type processes with marginal diistributions $\Gamma (\alpha_k,\beta)$, $k\geq 1$ where
\begin{equation*}
\alpha _{k}=\alpha k^{-(1+2(1-H))}, \quad \frac{1}{2}<H<1,
\end{equation*}%
then condition (A) is satisfied with $\delta_k=\alpha_k$, and if $\sum_{k=1}^\infty \alpha _k<\infty$, condition (B) is satisfied as well. The supOU process
$X_{\infty }(t)=\sum_{k=1}^{\infty }X^{(k)}(t),\,t \geq 0$
has a marginal $\Gamma (\sum_{k=1}^{\infty }\alpha_{k},\beta )$ distribution.
\end{example}

\begin{example}
The stationary inverse Gaussian OU type process has the cumulant generating function
\begin{equation*}
\kappa (\zeta)=\log \E \exp \left\{ i\zeta X(t) \right\} = \delta \left( \gamma - \sqrt{\gamma^2-2i\zeta}\right) = \sum_{m=1}^{\infty} \frac{\delta (2m)! (i\zeta )^m}{(2m-1)(m!)^2 2^m \gamma^{2m-1}},
\end{equation*}%
$\gamma >0$, $\delta >0$. It follows that independent stationary OU type processes $\{X^{(k)}(t), t\geq 0\}$, $k\geq 1$ with marginals $IG (\delta _{k},\gamma)$, $k\geq 1$ where
\begin{equation*}
\delta_{k}=\delta k^{-(1+2(1-H))}, \quad \frac{1}{2}<H<1,
\end{equation*}%
satisfy conditions (A) and (B), and we obtain inverse Gaussian supOU process
\begin{equation*}
X_{\infty }(t)=\sum_{k=1}^{\infty }X^{(k)}(t), \quad t \geq 0,
\end{equation*}%
with marginal $IG (\sum_{k=1}^{\infty }\delta _{k},\gamma)$ distribution.
\end{example}

\begin{example}
The stationary Variance Gamma OU type process has the
the cumulant generating function
\begin{equation*}
\kappa (\zeta)=\log \E \exp \left\{i \zeta X(t) \right\} = i\mu \zeta + 2 \kappa \log \left( \frac{\gamma}{\alpha^2-(\beta+i\zeta)^2} \right),
\end{equation*}
$\kappa >0$, $\alpha >|\beta |>0$, $\mu  \in \rr$, $\gamma ^2=\alpha ^2-\beta ^2$, $\vert \beta + \zeta \vert < \alpha $.
It follows that $VG\left(\kappa,\alpha ,\beta ,\mu \right)$ distribution is closed under convolution with respect to parameters $\kappa$ and $\mu$. Independent stationary OU type processes $\{X^{(k)}(t), t\geq 0\}$, $k\geq 1$ with marginals $VG\left(\kappa_k,\alpha ,\beta ,\mu_k \right)$, $k\geq 1$ where $\sum _{k=1}^\infty \mu _k$ converges and
\begin{equation*}
\kappa_{k}=\kappa k^{-(1+2(1-H))}, \quad
\frac{1}{2}<H<1,
\end{equation*}%
satisfy conditions (A) and (B), and we obtain variance gamma supOU process
\begin{equation*}
X_{\infty }(t)=\sum_{k=1}^{\infty }X^{(k)}(t), \quad t \geq 0,
\end{equation*}%
with marginal $VG\left(\sum_{k=1}^{\infty }\kappa_{k},\alpha ,\beta ,\sum_{k=1}^{\infty }\mu_{k} \right)$ distribution.
\end{example}

\begin{example}
The stationary normal inverse Gaussian OU type process
has cumulant generating function
\begin{equation*}
\kappa (\zeta)=\log \E \exp \left\{ i\zeta X(t) \right\}  =i\mu \zeta +\delta \left( \sqrt{\alpha^{2}-\beta ^{2}}-\sqrt{\alpha ^{2}-\left( \beta +i\zeta \right) ^{2}}\right),
\end{equation*}%
$\alpha \geq \left\vert \beta \right\vert \geq 0,$  $\delta >0$, $\mu \in \mathbb{R}$, $\left\vert \beta +\zeta \right\vert <\alpha $. It follows that $NIG(\alpha ,\beta ,\delta ,\mu )$ distribution is closed under convolution with respect to parameters $\delta$ and $\mu$. Independent stationary OU type processes $\{X^{(k)}(t), t\geq 0\}$, $k\geq 1$ with marginals $NIG(\alpha ,\beta ,\delta_k ,\mu_k )$, $k\geq 1$ with convergent $\sum_{k=1}^\infty \mu _k$,
\begin{equation*}
\delta_{k}=\delta k^{-(1+2(1-H))}, \quad
\frac{1}{2}<H<1,
\end{equation*}%
satisfy conditions (A) and (B),  and we obtain normal inverse Gaussian supOU process
\begin{equation*}
X_{\infty }(t)=\sum_{k=1}^{\infty }X^{(k)}(t), \quad t \geq 0,
\end{equation*}%
with marginal $NIG(\alpha ,\beta ,\sum_{k=1}^{\infty }\delta_{k}, \sum_{k=1}^{\infty } \mu_k )$ distribution.
\end{example}

\begin{example}
The stationary positive tempered stable OU type process has the cumulant generating function
\begin{equation*}
\kappa (\zeta)=\log \E \exp \left\{ i\zeta X(t) \right\} =\delta \gamma -\delta \left( \gamma ^{\frac{1}{\kappa }}-2i\zeta \right) ^{\kappa },
\end{equation*}
$\kappa \in (0,1)$, $\delta >0$, $\gamma >0$, $0<\zeta <\frac{\gamma ^{1/\kappa }}{2}$.
Thus the $TS(\kappa,\delta ,\gamma )$ distribution is closed under convolution with respect to parameter $\delta$. Independent stationary OU type processes $\{X^{(k)}(t), t\geq 0\}$, $k\geq 1$ with marginals $TS(\kappa,\delta_k ,\gamma )$, $k\geq 1$ where
\begin{equation*}
\delta_{k}=\delta k^{-(1+2(1-H))}, \quad \frac{1}{2}<H<1,
\end{equation*}%
satisfy conditions (A) and (B), and we obtain tempered stable supOU process
\begin{equation*}
X_{\infty }(t)=\sum_{k=1}^{\infty }X^{(k)}(t), \quad t \geq 0,
\end{equation*}%
with marginal $TS(\kappa,\sum_{k=1}^{\infty }\delta_{k} ,\gamma )$ distribution.
\end{example}

More examples of supOU type processes satisfying Condition (A) can be derived from other distributions, for example, normal tempered stable, Euler's gamma distribution and $z$-distribution.

\bibliographystyle{abbrv}
\bibliography{References}

\end{document}